\documentclass[12pt,reqno]{amsproc}
\usepackage[margin=1in]{geometry}

\usepackage{type1cm}        
\usepackage{amsmath,amsthm,amssymb}         

\usepackage{graphicx}        
\usepackage{multicol}        
\usepackage[bottom]{footmisc}
\usepackage{url}

\title{Canonical Stratifications along Bisheaves}

\author{Vidit Nanda}
\email{nanda@maths.ox.ac.uk}
\author{Amit Patel}
\email{akpatel@colostate.edu}

\usepackage{nicefrac, graphicx, mathrsfs}
\usepackage{amsmath, amssymb, amsthm}

\usepackage[sc]{mathpazo}\renewcommand{\mathbf}{\mathbold}

\usepackage{tikz}
\usetikzlibrary{matrix,arrows}

\usepackage[all,2cell,cmtip]{xy}
\UseAllTwocells

\renewcommand{\mathcal}{\mathscr}

\usepackage{hyperref}
\hypersetup
{
    colorlinks,	%
    citecolor=red,%
    filecolor=black,%
    linkcolor=blue,%
    urlcolor=blue,
}

\theoremstyle{plain}
  \newtheorem{theorem}{Theorem}[section]
  \newtheorem{corollary}[theorem]{Corollary}
  \newtheorem{lemma}[theorem]{Lemma}
  
  \newtheorem{proposition}[theorem]{Proposition}

\theoremstyle{definition}
  \newtheorem{definition}[theorem]{Definition}
  
  \newtheorem{ex}[theorem]{Example}
  
  \newtheorem{remark}[theorem]{Remark}
  \newenvironment{example}{\begin{ex}}{\end{ex}}

  \newcommand{\R}{\mathbb{R}}

  \newcommand{\Z}{\mathbb{Z}}

  \newcommand{\category}[1]{\mathrm{\mbox{\bf #1}}}
	\newcommand{\Face}{{\category{Fc}}}

	\newcommand{\Ab}{\category{Ab}}

	\newcommand{\bH}{\text{H}}

	\newcommand{\cU}{\mathcal{U}}

	\newcommand{\HBM}{\bH^\text{\tiny BM}}
	
	\newcommand{\inj}{\hookrightarrow}

\newcommand{\st}{\category{st }}

	\newcommand{\X}{\mathcal{X}}
	\newcommand{\M}{\mathcal{M}}
	\newcommand{\K}{\mathcal{K}}

	\newcommand{\Fc}[1]{\Fc[#1]} 

\newcommand{\shf}[1]{\overline{\mathbf{#1}}}
\newcommand{\csh}[1]{\underline{\mathbf{#1}}}
\newcommand{\bsh}[1]{\overline{\underline{\mathbf{#1}}}}

\newcommand{\leqdn}{\mathrel{\text{\rotatebox[origin=c]{-90}{$\leq$}}}}
\newcommand{\roteq}{\mathrel{\text{\rotatebox[origin=c]{90}{$=$}}}}

\date{}

\begin{document}

\maketitle

\begin{abstract}A theory of bisheaves has been recently introduced to measure the homological stability of fibers of maps to manifolds. A bisheaf over a topological space is a triple consisting of a sheaf, a cosheaf, and compatible maps from the stalks of the sheaf to the stalks of the cosheaf. In this note we describe how, given a bisheaf constructible (i.e., locally constant) with respect to a triangulation of its underlying space, one can explicitly determine the coarsest stratification of that space for which the bisheaf remains constructible.\end{abstract}

\section*{Introduction}

The space of continuous maps from a compact topological space $\X$ to a metric space $\M$ carries a natural metric structure of its own --- the distance between $f,g:\X\to\M$ is given by $\sup_{x \in \X} d_\M[f(x),g(x)]$, where $d_\M$ is the metric on $\M$. It is natural to ask how sensitive the fibers $f^{-1}(p)$ over points $p \in \M$ are to perturbations of $f$ in this metric space of maps $\X\to\M$. The case $\M= \R$ (endowed with its standard metric) is already interesting, and lies at the heart of both Morse theory \cite{milnor} and the stability of persistent homology \cite{interlevelsets, cdsgo, cseh}.

The theory of {\bf bisheaves} was introduced in \cite{macpat} to provide stable lower bounds on the homology groups of such fibers in the case where $f$ is a reasonably tame (i.e., Thom-Mather stratified) map. The fibers of $f$ induce two algebraic structures generated by certain basic open subsets $\cU \subset \M$ --- their Borel-Moore homology $\HBM_\bullet(f^{-1}(\cU)) =  \bH_\bullet(\X,\X-f^{-1}(\cU))$ naturally forms a sheaf of abelian groups, whereas their singular homology $\bH_\bullet(f^{-1}(\cU))$ naturally forms cosheaf. If $\M$ is a $\Z$-orientable manifold, then its fundamental class---let's call it $o \in \bH^m_c(\M)$---restricts to a generator $o_\cU$ of the top compactly-supported cohomology $\bH^m_c(\cU)$ of basic open subsets $\cU \subset \M$. The cap product \cite[Sec 3.3]{hatcher} with its pullback $f^*(o_\cU) \in \bH^m_c(f^{-1}(\cU))$ therefore induces group homomorphisms
\[
\HBM_{m+\bullet}(f^{-1}(\cU)) {\longrightarrow} \bH_\bullet(f^{-1}(\cU))
\]
from the ($m$-shifted) Borel-Moore to the singular homology over $\cU$. These maps commute with restriction maps of the sheaf and extension maps of the cosheaf by naturality of the cap product. This data, consisting of a sheaf plus a cosheaf along with such maps is the prototypical and motivating example of a bisheaf. 

Fix an arbitrary open set $\cU \subset \M$ and restrict the bisheaf described above to $\cU$. 
We replace the restricted Borel-Moore sheaf with its largest sub {\em episheaf} 
(i.e., a sheaf whose restriction maps on basic opens are all surjective), and similarly, we replace the restricted singular cosheaf with its largest quotient {\em monocosheaf} (i.e., a cosheaf whose extension maps on basic opens are all injective).
It is not difficult to confirm that even after the above alterations, one can induce canonical maps from the episheaf to the monocosheaf which form a new bisheaf over $\cU$. The stalkwise-images of the maps from the episheaf to the monocosheaf in this new bisheaf form a {\em local system} over $\cU$ --- this may be viewed as either a sheaf or a cosheaf depending on taste, since all of its restriction/extension maps are invertible. The authors of \cite{macpat} call this the {\bf persistent local system} of $f$ over $\cU$.
The persistent local system of $f$ over $\cU$ is a collection of subquotients of $\bH_\bullet(f^{-1}(p))$ for all $p \in \cU$
and provides a principled lower bound for the fiberwise homology of $f$ over $\cU$ which is stable to perturbations. For a sufficiently small $\epsilon > 0$, let $\cU_\epsilon$ be the shrinking of $\cU$ by $\epsilon$. For all tame maps $g: \X \to \M$ within $\epsilon$ of $f$, the persistent local system of $f$ over $U$ restricted to $\cU_\epsilon$ is a fiberwise subquotient of the persistent local system of $g$ over $\cU_\epsilon$.

The goal of this paper is to take the first concrete steps towards rendering this new theory of bisheaves amenable to explicit machine computation. In Sec \ref{sec:simpbish} we introduce the notion of a {\bf simplicial bisheaf}, i.e., a bisheaf which is constructible with respect to a fixed triangulation of the underlying manifold $\M$. Such bisheaves over simplicial complexes are not much harder to represent on computers than the much more familiar cellular (co)sheaves --- if we work with field coefficients rather than integers, for instance, a simplicial bisheaf amounts to the assignment of one matrix to each simplex $\sigma$ of $\M$ and two matrices to each face relation $\sigma \leq \sigma'$, subject to certain functoriality constraints --- more details can be found in Sec \ref{sec:simpbish} below. 

On the other hand, bisheaves are profoundly different from (co)sheaves in certain fundamental ways --- as noted in \cite{macpat}, the category of bisheaves, simplicial or otherwise, over a manifold $\M$ is not abelian. Consequently, we have no direct recourse to bisheafy analogues of basic (co)sheaf invariants such as sheaf cohomology and cosheaf homology. Even so, some of the ideas which produced efficient algorithms for computing cellular sheaf cohomology \cite{cgn} can be suitably adapted towards the task of extracting the persistent local system from a given simplicial bisheaf. One natural way to accomplish this is to find the coarsest partition of the simplices of $\M$ into regions so that over each region the cap product map relating the Borel-Moore stalk to the singular costalk is locally constant. This idea is made precise  in Sec \ref{sec:strat}.

Our main construction is described in Sec \ref{sec:main}. Following \cite{strat}, we use the bisheaf data over an $m$-dimensional simplicial complex $\M$ to explicitly construct a stratification by simplicial subcomplexes 
\[
\varnothing = \M_{-1} \subset \M_0 \subset \cdots \subset \M_{m-1} \subset \M_m = \M,
\]
called the {\bf canonical stratification of $\M$} along the given bisheaf; the connected components of each $\M_d - \M_{d-1}$, called the {\em canonical $d$-strata}, enjoy three remarkably convenient properties for our purposes.
\begin{enumerate}
	\item {\em Constructibility}: if two simplices lie in the same stratum, then the cap-product maps assigned to them by the bisheaf are related by invertible transformations.
	\item {\em Homogeneity}: if two adjacent simplices $\sigma \leq \sigma'$ of $\M$ lie in different strata, then the (isomorphism class of the) bisheaf data assigned to the face relation $\sigma \leq \sigma'$ in $\M$ depends only on those strata.
	\item {\em Universality:} this is the coarsest stratification (i.e., the one with fewest strata) satisfying both constructibility and homogeneity. 
\end{enumerate} 

Armed with the canonical stratification of $\M$ along a bisheaf, one can reduce the computational burden of building the associated persistent local system as follows. Rather than extracting an episheaf and monocosheaf for {\em every} simplex and face relation, one only has to perform these calculations for each canonical stratum. The larger the canonical strata are, the more computationally beneficial this strategy becomes.

{\footnotesize
\subsection*{Acknowledgements} This work had its genesis in the {\em Abel Symposium on Topological Data Analysis}, held in June 2018 amid the breathtaking fjords of Geiranger, where both authors gave invited lectures. AP spoke about \cite{macpat} and VN about \cite{strat}, and it became clear to us almost immediately that there were compelling practical reasons to combine these works. It is a sincere pleasure to thank the Abel Foundation and the Abel Symposium organizers, particularly Nils Baas and Marius Thaule, for giving us the opportunity to work in such an inspiring location. The ideas of Robert MacPherson are densely sprinkled throughout not only this paper, but also across both its progenitors \cite{macpat} and \cite{strat}. We are grateful to the Institute for Advanced Study for hosting many of our discussions with Bob. We also thank the anonymous referee for encouraging us to clarify Def \ref{def:fstrat} and the subsequent remark.

VN’s work is supported by The Alan Turing Institute under the EPSRC grant number EP/N510129/1, and by the Friends of the Institute for Advanced Study. AP's work is supported by the National Science Foundation under agreement number CCF-1717159.
}

\section{Bisheaves around Simplicial Complexes} \label{sec:simpbish}

Let $\M$ be a simplicial complex and let $\Ab$ denote the category of abelian groups. By a {\bf {sheaf} over $\M$} we mean a functor
\[
\shf{F}:\Face(\M) \to \Ab
\] 
from the poset of simplices in $\M$ ordered by the face relation to the abelian category $\Ab$. In other words, each simplex $\sigma$ of $\M$ is assigned an abelian group $\shf{F}(\sigma)$ called the {\em stalk} of $\shf{F}$ over $\sigma$, while each face relation $\sigma \leq \sigma'$ among simplices is assigned a group homomorphism $\shf{F}(\sigma \leq \sigma'):\shf{F}(\sigma) \to \shf{F}(\sigma')$ called its {\em restriction map}. These assignments of objects and morphisms are constrained by the usual functor-laws of associativity and identity. A morphism $\shf{\alpha}:\shf{F} \to \shf{G}$ of sheaves over $\M$ is prescribed by a collection of group homomorphisms $\{\shf{\alpha}_\sigma:\shf{F}(\sigma) \to \shf{G}(\sigma)\}$, indexed by simplices of $\M$, which must commute with restriction maps. 

The dual notion is that of a {\bf {cosheaf} under $\M$}, which is a functor
\[
\csh{F}:\Face(\M)^\text{op} \to \Ab;
\] 
this assigns to each simplex $\sigma$ an abelian group $\csh{F}(\sigma)$ called its {\em costalk}, and to each face relation $\sigma \leq \sigma'$ a contravariant group homomorphism $\csh{F}(\sigma \leq \sigma'):\csh{F}(\sigma') \to \csh{F}(\sigma)$, called the {\em extension map}. As before, a morphism $\csh{\alpha}:\csh{F} \to \csh{G}$ of cosheaves under $\M$ is a simplex-indexed collection of abelian group homomorphisms  $\{\csh{\alpha}_\sigma:\csh{F}(\sigma) \to \csh{G}(\sigma)\}$ which must commute with extension maps. For a thorough introduction to cellular (co)sheaves, the reader should consult \cite{curry}.

\subsection{Definition}

The following algebraic-topological object (see \cite[Def 5.1]{macpat}) coherently intertwines sheaves with cosheaves. 
\begin{definition}
	\label{def:bisheaf}
	A {\bf {bisheaf} around $\M$} is a triple $\bsh{F} = (\shf{F},\csh{F},F)$ defined as follows. Here $\shf{F}$ is a sheaf over $\M$, while $\csh{F}$ is an cosheaf under $\M$, and 
	\[
	F = \{F_\sigma:\shf{F}(\sigma) \to \csh{F}(\sigma)\}
	\] is a collection of abelian group homomorphisms indexed by the simplices of $\M$ so that the following diagram, denoted $\bsh{F}(\sigma \leq \sigma')$, commutes for each face relation $\sigma \leq \sigma'$:
	\[
	\xymatrixcolsep{1in}
	\xymatrix{
		\shf{F}(\sigma) \ar@{->}[d]_{F_\sigma}\ar@{->}[r]^{\shf{F}(\sigma \leq \sigma')}& \shf{F}(\sigma') \ar@{->}[d]^{F_{\sigma'}} \\
		\csh{F}(\sigma) & \csh{F}(\sigma') \ar@{->}[l]^{\csh{F}(\sigma \leq \sigma')}\\
	}
	\]
	(The right-pointing map is the restriction map of the sheaf $\shf{F}$, while the left-pointing map is the extension map of the cosheaf $\csh{F}$).
\end{definition}

\subsection{Bisheaves from Fibers}

The following construction is adapted from \cite[Ex 5.3]{macpat}. Consider a map $f:\X\to\M$ whose target space $\M$ is a connected, triangulated manifold of dimension $m$. Let $o$ be a generator of the top compactly-supported cohomology group $\bH_\text{\tiny c}^m(\M)$. Our assumptions on $\M$ imply $\bH_\text{\tiny c}^m(\M) \simeq \Z$, so $o \in \{\pm 1\}$. Now the inclusion $\st \sigma \subset \M$ of the open star\footnote{The open star of $\sigma \in \M$ is given by $\st \sigma = \{\tau \in \M \mid \sigma \leq \tau\}$.} of any simplex $\sigma$ in $\M$ induces an isomorphism on $m$-th compactly supported cohomology, so let $o|_\sigma$ be the image of $o$ in $\bH^m_c(\st \sigma)$ under this isomorphism. Since $f$ restricts to a map $f^{-1}(\st \sigma) \to \st \sigma$, the generator $o|_\sigma$ pulls back to a class $f^*(o|_\sigma)$ in $\bH^m_c(f^{-1}(\st \sigma))$. The cap product with $f^*(o|_\sigma)$ therefore constitutes a map 
\[
\xymatrixcolsep{0.7in}
\xymatrix{
	\bH_{m+\bullet}^\text{\tiny BM}\left(f^{-1}(\st \sigma)\right) \ar@{->}[r]^-{\frown f^*(o|_\sigma)} & \bH_\bullet\left(f^{-1}(\st \sigma)\right)
}
\]
from the Borel-Moore homology to the singular homology of the fiber $f^{-1}(\st \sigma)$. We note that the former naturally forms a sheaf over $\M$ while the later forms a cosheaf; as mentioned in the Introduction, the above data constitutes the primordial example of a bisheaf.

\section{Stratifications along Bisheaves} \label{sec:strat}

Throughout this section, we will assume that $\bsh{F} = (\shf{F},\csh{F},F)$ is a bisheaf of abelian groups over some simplicial complex $\M$ of dimension $m$ in the sense of Def \ref{def:bisheaf}. We do not require this $\M$ to be a manifold.

\begin{definition}
	\label{def:fstrat}
An {\bf $\bsh{F}$-{stratification} of $\M$} is a filtration  $\K_\bullet$ by subcomplexes:
\[
\varnothing = \K_{-1} \subset \K_0 \subset \cdots \subset \K_{m-1} \subset \K_m = \M,
\]
so that connected components of the (possibly empty) difference $\K_d - \K_{d-1}$, called the {\em $d$-dimensional strata} of $\K_\bullet$, obey the following axioms.
\begin{enumerate}
		\item{\bf Dimension:} The maximum dimension of simplices lying in a $d$-stratum should precisely equal $d$ (but we do not require every simplex in a $d$-stratum $S$ to be the face of some $d$-simplex in $S$).
		
		\item{\bf Frontier:} The transitive closure of the following binary relation $\prec$ on the set of all strata forms a partial order: we say $S \prec S'$ if there exist simplices $\sigma \in S$ and $\sigma' \in S'$ with $\sigma \leq \sigma'$. Moreover, this partial order is graded in the sense that $S \prec S'$ implies $\dim S \leq \dim S'$, with equality of dimension occurring if and only if $S = S'$. 

		\item {\bf Constructibility:} $\bsh{F}$ is {locally constant} on each stratum. Namely, if two simplices $\sigma \leq \tau$ of $\M$ lie in the same stratum, then $\shf{F}(\sigma \leq \tau)$ and $\csh{F}(\sigma \leq \tau)$ are both isomorphisms.
\end{enumerate}
\end{definition}

\begin{remark}
		
	It follows from constructibility (and the fact that strata must be connected) that the commuting diagram $\bsh{F}(\sigma \leq \sigma')$ assigned to simplices $\sigma \leq \sigma'$ of $\M$ depends, up to isomorphism, only on the strata containing $\sigma$ and $\sigma'$. That is, given any other pair $\tau \leq \tau'$  so that $\sigma$ and $\tau$ lie in the same stratum $S$ while $\sigma'$ and $\tau'$ lie in the same stratum $S'$, there exist four isomorphisms (depicted as dashed vertical arrows) which make the following cube of abelian groups commute up to isomorphism:
	\[
	\xymatrixcolsep{.4in}
	\xymatrixrowsep{0.05in}
	\xymatrix{
		& \shf{F}(\sigma) \ar@{-->}[dddd]_\sim \ar@{->}[rr]^{\shf{F}(\sigma \leq \sigma')} \ar@{->}[dl]_{F_\sigma} & & \shf{F}(\sigma') \ar@{-->}[dddd]^\sim \ar@{->}[dr]^{F_{\sigma'}} &\\
		\csh{F}(\sigma)  & & & &  \csh{F}(\sigma') \ar@{->}[llll]^{\csh{F}(\sigma \leq \sigma')} \\
		& & & &  \\
		& & & & \\
		& \shf{F}(\tau) \ar@{->}[rr]^{\shf{F}(\tau \leq \tau')} \ar@{->}[dl]_{F_\tau} & & \shf{F}(\tau') \ar@{->}[dr]^{F_{\tau'}} &\\
		\csh{F}(\tau) \ar@{-->}[uuuu]^\sim & & & &  \csh{F}(\tau') \ar@{->}[llll]^{\csh{F}(\tau \leq \tau')} \ar@{-->}[uuuu]_\sim
	}
	\]	
	These vertical isomorphisms are not unique, but rather depend on choices of paths lying in $S$ (from $\sigma$ to $\tau$) and in $S'$ (from $\sigma'$ to $\tau'$).
\end{remark}

\begin{example} The first example of an $\bsh{F}$-stratification of $\M$ that one might consider is the {\bf skeletal} stratification, where the $d$-strata are simply the $d$-simplices. 
\end{example}
	
Since we are motivated by computational concerns, we seek an $\bsh{F}$-stratification with as few strata as possible. To make this notion precise, note that the set of all $\bsh{F}$- stratifications of $\M$ admits a partial order --- we say that $\K_\bullet$ {\em refines} another $\bsh{F}$-stratification $\K'_\bullet$ if every stratum of $\K_\bullet$ is  contained inside some stratum of $\K'_\bullet$ (when both are viewed as subspaces of $\M$). The skeletal stratification refines all the others, and serves as the maximal object in this poset; and the object that we wish to build here lies at the other end of this hierarchy.

\begin{definition}
	\label{def:canstr}
The {\bf canonical} $\bsh{F}$-stratification of $\M$ is the minimal object in the poset of $\bsh{F}$-stratifications of $\M$ ordered by refinement --- every other stratification is a refinement of the canonical one.
\end{definition}

The reader may ask why this object is well-defined at all --- why should the poset of all $\bsh{F}$-stratifications admit a minimal element, and even if it does, why should that element be unique? Taking this definition as provisional for now, we will establish the existence and uniqueness of the {canonical $\bsh{F}$-stratification} of $\M$ via an explicit construction in the next section.

\section{The Main Construction} \label{sec:main}

As before, we fix a bisheaf $\bsh{F} = (\shf{F},\csh{F},F)$ on an $m$-dimensional simplicial complex $\M$. Our goal is to construct the canonical $\bsh{F}$-stratification, which was described in Def \ref{def:canstr} and will be denoted here by $\M_\bullet$:
\[
\varnothing = \M_{-1} \subset \M_0 \subset \cdots \subset \M_{m-1} \subset \M_m = \M.
\] We will establish the existence and uniqueness of this stratification by constructing the strata in reverse-order: the $m$-dimensional canonical strata will be identified before the $(m-1)$-dimensional canonical strata, and so forth. There is a healthy precedent for such top-down constructions that dates back to work of Whitney \cite{whitney} and Goresky-MacPherson \cite[Sec 4.1]{gm2}.

\subsection{Localizations of the Face Poset}
The key ingredient here, as in \cite{strat}, is the ability to {\em localize} \cite[Ch I.1]{gabzis} the poset $\Face(\M)$ about a special sub-collection $W$ of face relations that is {closed} in the following sense: if $(\sigma \leq \tau)$ and $(\tau \leq \nu)$ both lie in $W$ then so does $(\sigma \leq \nu)$.

\begin{definition}
\label{def:loc}
	Let $W$ be a closed collection of face relations in $\Face(\M)$ and let $W^+$ denote the union of $W$ with all equalities of the form $(\sigma = \sigma)$ for $\sigma$ ranging over simplices in $\M$. The {\bf localization} of $\Face(\M)$ about $W$ is a category $\Face_W(\M)$ whose objects are the simplices of $\M$, while morphisms from $\sigma$ to $\tau$ are given by equivalence classes of finite (but arbitrarily long) $W$-{\em zigzags}. These have the form
	\[
	(\sigma \leq \tau_0 \geq \sigma_0 \leq \cdots \leq \tau_k \geq \sigma_k \leq \tau), \text{ where:}
	\]
	\begin{enumerate}
		\item only relations in $W^+$ can point backwards (i.e., $\geq$),
		\item composition is given by concatenation, and
		\item the trivial zigzag $(\sigma = \sigma)$ represents the identity morphism of each simplex $\sigma$.
	\end{enumerate}
	The equivalence between $W$-zigzags is generated by the transitive closure of the following basic relations. Two such zigzags are related
	\begin{itemize}
		\item {\em horizontally} if one is obtained from the other by removing internal equalities, e.g.:
		\begin{align*}
		\left(\cdots  \leq \tau_0 \geq \sigma_0 = \sigma_0 \geq \tau_1 \leq \cdots \right) &\sim \left(\cdots  \leq \tau_0 \geq \tau_1 \leq \cdots\right), \\ 
		\left(\cdots  \geq \sigma_0 \leq \tau_1 = \tau_1 \leq \sigma_1 \geq \cdots \right) &\sim \left(\cdots  \geq \sigma_0 \leq \sigma_1 \geq \cdots\right),  
		\end{align*}
		\item or {\em vertically}, if they form the rows of a grid:
		\begin{alignat*}{11}
		\sigma~ & ~\leq~ &   ~\tau_0~  &  ~\geq~ &    ~\sigma_0~   &   ~\leq~ & ~\cdots~ & ~\geq~ &    ~\sigma_k~   &  ~\leq~ & ~\tau \\
		\roteq~  &              & \leqdn~  &              &  \leqdn~ &                &                &             & \leqdn~    &               &   \roteq    \\
		\sigma~ & ~\leq~ &   ~\tau'_0~  &  ~\geq~ &     ~\sigma'_0~  &  ~\leq~  &  ~\cdots~ & ~\geq~ &     ~\sigma'_k~ &   ~\leq~ & ~\tau
		\end{alignat*}
		whose vertical face relations (also) lie in $W^+$.
	\end{itemize} 
\end{definition}

\begin{remark}
These horizontal and vertical relations are designed to render invertible all the face relations $(\sigma \leq \tau)$ that lie in $W$. The backward-pointing $\tau \geq \sigma$ which may appear in a $W$-zigzag serves as the formal inverse to its forward-pointing counterpart $\sigma \leq \tau$ --- one can use a vertical relation followed by a horiztonal relation to achieve the desired cancellations whenever $(\cdots \geq \sigma \leq \tau \geq \sigma \leq \cdots)$ or $(\cdots \leq \tau \geq \sigma \leq \tau \geq \cdots)$ are encountered as substrings of a $W$-zigzag.   
\end{remark}

\subsection{Top Strata}
Consider the subset of face relations in $\Face(\M)$ to which $\bsh{F}$ assigns invertible maps, i.e.,
\begin{align}
\label{eq:E}
E = \{(\sigma \leq \tau) \text{ in } \Face(\M) \mid \shf{F}(\sigma \leq \tau) \text{ and }\csh{F}(\sigma \leq \tau) \text{ are isomorphisms}\}.
\end{align}
One might expect, in light of the constructibility requirement of Def \ref{def:fstrat}, that finding canonical strata would amount to identifying isomorphism classes in the localization of $\Face(\M)$ about $E$. Unfortunately, this does not work --- the pieces of $\M$ obtained in such a manner do not obey the frontier axiom in general. To rectify this defect, we must suitably modify $E$. Define the set of simplices
\begin{align*}
U = \{\sigma \in \Face(\M) \mid (\sigma \leq \tau) \in E \text{ for all } \tau \in \st \sigma\},
\end{align*}
and consider the subset $W \subset E$ given by
\begin{align}
\label{eq:W}
W = \{(\sigma \leq \tau) \in E \mid \sigma \in U\}.
\end{align}
Thus, a pair of adjacent simplices $(\sigma \leq \tau)$ of $\M$ lies in $W$ if and only if the sheaf $\shf{F}$ and cosheaf $\csh{F}$ assign isomorphisms not only to $(\sigma \leq \tau$) itself, but also to {\em all other face relations} encountered among simplices in the open star of $\sigma$. For our purposes, it is important to note that $U$ is {\em upward closed} as a subposet of $\Face(\M)$, meaning that $\sigma \in U$ and $\sigma' \geq \sigma$ implies $\sigma' \in U$.

\begin{proposition}
	\label{prop:isostrat}
Every simplex $\tau$ lying in an $m$-stratum of any $\bsh{F}$-stratification of $\M$ must be isomorphic in $\Face_{W}(\M)$ to an $m$-dimensional simplex of $\M$. 
\end{proposition}
\begin{proof}
Assume $\tau$ lies in an $m$-dimensional stratum $S$ of an $\bsh{F}$-stratification of $\M$. By the dimension axiom, $S$ contains at least one $m$-simplex, which we call $\sigma$. Since $S$ is connected,  there exists a zigzag of simplices lying entirely in $S$ that links $\sigma$ to $\tau$, say
\[
\zeta = (\sigma \leq \tau_0 \geq \sigma_0 \leq \cdots \leq \tau_k \geq \sigma_k \leq \tau).
\]
By the constructibility requirement of Def \ref{def:fstrat}, every face relation in sight (whether $\leq$ or $\geq$) lies in $E$. And by the frontier requirement of that same definition, membership in $m$-strata is upward closed, so in particular all the $\sigma_\bullet$'s lie in $U$. Finally, since $\sigma$ is top-dimensional and $\tau_0 \geq \sigma$, we must have $\tau_0 = \sigma$. Thus, not only is our $\zeta$ a $W$-zigzag, but it also represents an invertible morphism in $\Face_{W}(\M)$. Indeed, a $W$-zigzag representing its inverse can be obtained simply by traversing backwards:
\[
\zeta^{-1} = (\tau \leq \tau \geq \sigma_k \leq \tau_k \geq \cdots \leq \tau_0 \geq \sigma \leq \sigma).
\]
This confirms that $\sigma$ and $\tau$ are isomorphic in $\Face_{W}(\M)$, as desired.
\end{proof}

Given the preceding result, the coarsest $m$-strata that one could hope to find are isomorphism classes of $m$-dimensional simplices in $\Face_{W}(\M)$. 

\begin{proposition}
\label{prop:topstrat}
The canonical $m$-strata of $\M_\bullet$ are precisely the isomorphism classes of $m$-dimensional simplices in $\Face_{W}(\M)$.
\end{proposition}
\begin{proof}
Let $\sigma$ be an $m$-simplex of $\M$. We will show that the set $S$ of all $\tau$ which are isomorphic to $\sigma$ forms an $m$-stratum by verifying the frontier and constructibility axioms from Def \ref{def:fstrat} --- the dimension axiom is trivially satisfied since $\sigma \in S$. Note that for any $\tau \in S$ there exists some $W$-zigzag whose simplices all lie in $S$, and which represents an isomorphism from $\sigma$ to $\tau$ in $\Face_{W}(\M)$. (The existence of these zigzags shows that $S$ is connected). So let us fix for each $\tau \in S$ such a zigzag
\[
\zeta_\tau = (\sigma \leq \tau_0 \geq \sigma_0 \leq \cdots \geq \sigma_k \leq \tau),
\]
and assume it is horizontally reduced in the sense that none of its order relations (except possibly the first and last $\leq$) are equalities. Thus, all the $\sigma_d$'s in $\zeta_\tau$ lie in $U$. Upward closure of $U$ now forces simplices in $\st \sigma_k$, which contains $\st \tau$, to also lie in $S$. This shows that $S$ satisfies the frontier axiom, because any simplex of $\M$ with a face in $S$ must itself lie in $S$. We now turn to establishing constructibility. Since $\sigma$ is top-dimensional, we know that $\tau_0 = \sigma$, so in fact the first $\leq$ in $\zeta_\tau$ must be an equality. Consider the bisheaf data $\bsh{F}(\zeta_\tau)$ living over our zigzag:
\[
\xymatrix{
 \shf{F}(\sigma) \ar@{->}[r] \ar@{->}[d] & \shf{F}(\tau_0) \ar@{->}[d] & \shf{F}(\sigma_0)  \ar@{->}[d] \ar@{->}[r] \ar@{->}[l] & \cdots \ar@{->}[r] & \shf{F}(\tau_k) \ar@{->}[d] & \shf{F}(\sigma_k) \ar@{->}[d] \ar@{->}[l]\ar@{->}[r]& \shf{F}(\tau) \ar@{->}[d]\\
 \csh{F}(\sigma) & \csh{F}(\tau_0) \ar@{->}[l]\ar@{->}[r]& \csh{F}(\sigma_0) & \cdots \ar@{->}[l] & \csh{F}(\tau_k)\ar@{->}[l] \ar@{->}[r] & \csh{F}(\sigma_k) &\ar@{->}[l]\csh{F}(\tau)
}
\]
(All horizontal homomorphisms in the top row are restriction maps of $\shf{F}$, all horizontal homomorphisms in the bottom row are extension maps of $\csh{F}$, and the vertical morphism in the column of a simplex $\nu$ is $F_\nu$). By definition of $W$ (and the fact that $\sigma = \tau_0)$, all horizontal maps in sight are isomorphisms, so in particular we may replace all left-pointing arrows in the top row and all the right-pointing arrows in the bottom row by their inverses to get abelian group isomorphisms $\phi_\tau:\shf{F}(\sigma) \to \shf{F}(\tau)$ and $\psi_\tau:\csh{F}(\tau) \to \csh{F}(\sigma)$ that fit into a commuting square with $F_\sigma$ and $F_\tau$. Now given any other simplex $\tau' \geq \tau$ lying in $S$, one can repeat the argument above with the bisheaf data $\bsh{F}\left(\zeta_{\tau'} \circ \zeta_\tau^{-1}\right)$ to confirm that
\[ \shf{F}(\tau \leq \tau') = \phi_{\tau'} \circ \phi_\tau^{-1} \quad \text{and} \quad 
\csh{F}(\tau \leq \tau') =  \psi_\tau^{-1} \circ \psi_{\tau'}.
\]
Thus both maps are isomorphisms, as desired.
\end{proof}

\subsection{Lower Strata}

Our final task is to determine which simplices lie in canonical strata of dimension $< m$. This is accomplished by iteratively modifying both the simplicial complex $\M = \M_m$ and the set of face relations $W = W_m$ which was defined in (\ref{eq:W}) above. 

\begin{definition}
	\label{def:subpair}
Given $d \in \{0,1,\ldots,m-1\}$, assume we have the pair $(\M_{d+1},W_{d+1})$ consisting of a simplicial complex $\M_{d+1}$ of dimension $\leq (d+1)$ and a collection $W_{d+1}$ of face relations in $\Face(\M)$. The subsequent pair $(\M_{d},W_{d})$ is defined as follows.
\begin{enumerate}
	\item The set $\M_{d}$ is obtained from $\M_{d+1}$ by removing all the simplices which are isomorphic to some $(d+1)$-simplex in the localization $\Face_{W_{d+1}}(\M)$.
	
	\item To define $W_{d}$, first consider the collection of simplices
	\[
	U_{d} = \{\sigma \in \Face(\M_{d}) \mid (\sigma \leq \tau) \in E \text{ for all } \tau \in \textbf{st}_{d}~\sigma\};
	\]
	here $\textbf{st}_{d}~\sigma$ is the open star of $\sigma$ in $\M_{d}$ (i.e., the collection of all $\tau \in \M_{d}$ satisfying $\tau \geq \sigma$), while $E$ is the set of face relations defined in (\ref{eq:E}). Now, set
	\[
	W_{d} = W_{d+1} \cup \{(\sigma \leq \tau) \mid \sigma \in U_{d} \text{ and }\tau \in \textbf{st}_{d}~\sigma\}.
	\]
\end{enumerate}
\end{definition}

\begin{proposition}
The sequence $\M_\bullet$ described in Def \ref{def:subpair} constitutes a filtration of the original simplicial complex $\M$ by subcomplexes with the property that $\dim \M_d \leq d$ for each $d \in \{0,1,\ldots,m\}$. 
\end{proposition}
\begin{proof}
Since $\M_m = \M$ is manifestly its own $m$-dimensinal subcomplex, it suffices by induction to show that if $\M_{d+1}$ is a simplicial complex of dimension $\leq (d+1)$, then the simplices in $\M_{d} \subset \M_{d+1}$ constitute a subcomplex of dimension $\leq d$. To this end, we will confirm that the difference $\M_{d+1} - \M_{d}$ satisfies two properties --- it must:
\begin{itemize}
	\item contain all the $(d+1)$-simplices in $\M_{d+1}$, and
	\item be upward closed with respect to the face partial order of $\M_{d+1}$.
\end{itemize}  Since every $(d+1)$-simplex is isomorphic to itself in $\Face_{W_{d+1}}(\M)$ via the identity morphism, the first requirement is immediately met. And by definition of $W_{d+1}$, if an arbitrary simplex $\sigma$ of $\M_{d+1}$ is isomorphic to a $(d+1)$-simplex in $\Face_{W_{d+1}}(\M)$, then so are all the simplices that lie in its open star $\textbf{st}_{d+1}~\sigma \subset \M_{d+1}$. Thus, our second requirement is also satisfied and the desired conclusion follows.   
\end{proof}

The structure of the sets $W_{\bullet}$ from Def \ref{def:subpair} enforces a convenient monotonicity among morphisms in the localization $\Face_{W_\bullet}(\M)$.
\begin{lemma}
	\label{lem:mono}
For each $d \in \{0,1,\ldots,m\}$, there are no morphisms in the localization $\Face_{W_d}(\M)$ from any simplex $\sigma$ in the difference $\M - \M_d$ to a simplex $\tau$ of $\M_{d}$.
\end{lemma}
\begin{proof}
Any putative morphism from $\sigma$ to $\tau$ in $\Face_{W_d}(\M)$ would have to be represented by a $W_{d}$-zigzag, say
\[
\zeta = (\sigma \leq \tau_0 \geq \sigma_0 \leq \cdots \leq \tau_k \geq \sigma_k \leq \tau).
\] Note that all face relations appearing here, except possibly the first $(\sigma \leq \tau_0)$, must lie in $W_{d}$ by upward closure. Since $\sigma \in \M - \M_d$, it must be isomorphic in $\Face_{W_{i}}(\M)$ to an $i$-simplex in $\M_{i}$ for some $i > d$. But the very existence of a zigzag representing such an isomorphism requires the bisheaf $\bsh{F}$ to be constant on the open star $\textbf{st}_{i}~\sigma$, meaning that $(\sigma \leq \tau_0)$ must lie in $W_{i} \subset W_d$. Thus, all the face relations ($\leq$ and $\geq$) encountered in $\zeta$ lie in $W_d$, whence $\zeta$ must be an isomorphism in the localization $\Face_{W_d}(\M)$ (with its inverse being given by backwards traversal). But now, $\tau$ would also be isomorphic to some $i$-simplex in $\Face_{W_{i}}(\M)$ with $i > d$, which forces the contradiction $\tau \not\in \M_d$. 
\end{proof}
 
Here is our main result.

\begin{theorem}
The sequence $\M_\bullet$ of simplicial complexes described in Def \ref{def:subpair} is the canonical $\bsh{F}$-stratification of $\M$. Moreover, for each $d \in \{0,1,\ldots,m\}$, the canonical $d$-strata of $\M_\bullet$ are isomorphism classes of $d$-simplices from $\M_d$ in the localization $\Face_{W_d}(\M)$.
\end{theorem}
\begin{proof}
We proceed by reverse-induction on $d$, with the base case $d=m$ being given by Prop \ref{prop:topstrat}. So we assume that the statement holds up to $(d+1)$, and establish that the canonical $d$-strata must be isomorphism classes of $d$-simplices from $\M_d$ in the localization $\Face_{W_d}(\M)$. Let $S$ denote the isomorphism class of a $d$-simplex $\sigma_*$ in $\M_d$. We will establish that $S$ satisfies all three axioms of Def \ref{def:fstrat}.
\begin{itemize}
\item {\bf Dimension:} clearly, $S$ contains a simplex $\sigma_*$ of dimension $d$; moreover, since $\dim \M_d \leq d$, all simplices of $\M$ with dimension $> d$ lie in $\M - \M_{d}$. None of these can be isomorphic in $\Face_{W_d}(\M)$ to $\sigma_*$ without contradicting Lem \ref{lem:mono}. 
\item {\bf Frontier:} it suffices to check antisymmetry of the relation $\prec$: there should be no simplices $\sigma \leq \sigma'$ with $\sigma \in \M-\M_d$ and $\sigma' \in S$. But the existence of such a $\sigma \leq \sigma'$ would result in a $W_d$-zigzag from $\sigma$ to $\sigma_*$, which is prohibited by Lem \ref{lem:mono}.
\item {\bf Constructibility:} it is straightforward to adapt the argument from the proof of Prop \ref{prop:topstrat} --- given simplices $\tau \leq \tau'$ both in $S$, one can find $W_d$-zigzags from $\sigma_*$ to $\tau$ and to $\tau'$ which guarantee that $\shf{F}(\tau \leq \tau')$ and $\csh{F}(\tau \leq \tau')$ are both isomorphisms.
\end{itemize}
To confirm that the strata obtained in this fashion are canonical, one can re-use the argument form the proof of Prop \ref{prop:isostrat} to show that a simplex which lies in a $d$-stratum of any $\bsh{F}$-stratification is isomorphic in $\Face_{W_d}(\M)$ to a $d$-simplex from $\M_d$, meaning that the  strata can not be any larger than these isomorphism classes.
\end{proof}

Finally, we remark that since the sets $W_\bullet$ defined in \ref{def:subpair} form a sequence that increases as $d$ decreases, the set of $W_d$-zigzags is contained in the set of $W_{d-1}$-zigzags and so forth. Therefore, successive localization of $\Face(\M)$ about these $W_\bullet$'s creates a nested sequence of categories:
\[
\Face_{W_m}(\M) \inj \Face_{W_{m-1}}(\M) \inj \cdots \inj \Face_{W_1}(\M) \inj \Face_{W_0}(\M).
\]
And thanks to the monotonicity guaranteed by Lem \ref{lem:mono}, isomorphism classes of $d$-simplices from $\M_d$ in $\Face_{W_d}(\M)$ are stable under inclusion to $\Face_{W_i}(\M)$ for $i \leq d$, since no simplex of $\M_i$ can ever become isomorphic to a simplex from $\M_d - \M_i$ in this entire sequence of categories. Consequently, we can extract all the canonical strata just by examining isomorphism classes in a single category.

\begin{corollary}
The $d$-dimensional strata of the canonical $\bsh{F}$-stratification of $\M$ are isomorphism classes of $d$-simplices from $\M_{d}$ in $\Face_{W_0}(\M)$.
\end{corollary}

\end{document}